\documentclass[a4paper]{amsart}
\usepackage{tikz-cd}
\usetikzlibrary{decorations.markings}
\makeatletter
\tikzcdset{
  open/.code     = {\tikzcdset{hook, circled};},
  closed/.code   = {\tikzcdset{hook, slashed};},
  open'/.code    = {\tikzcdset{hook', circled};},
  closed'/.code  = {\tikzcdset{hook', slashed};},
  circled/.code  = {\tikzcdset{markwith = {\draw (0,0) circle (.375ex);}};},
  slashed/.code  = {\tikzcdset{markwith = {\draw[-] (-.4ex,-.4ex) -- (.4ex,.4ex);}};},
  markwith/.code ={
    \pgfutil@ifundefined%
    {tikz@library@decorations.markings@loaded}%
    {\pgfutil@packageerror{tikz-cd}{You need to say %
      \string\usetikzlibrary{decorations.markings} to use arrows with markings}{}}{}%
    \pgfkeysalso{/tikz/postaction = {
      /tikz/decorate,
      /tikz/decoration={markings, mark = at position 0.5 with {#1}}}
    }
  },
}
\makeatother
\usepackage[T1]{fontenc}
\usepackage{amssymb,amsmath,latexsym,amsthm,paralist,mathrsfs,dsfont}
\usetikzlibrary{babel}
\usepackage{varioref}
\usepackage[pdfpagelabels]{hyperref}
\hypersetup{
  pdftitle={Homotopy invariance of tame homotopy groups of regular schemes},
  pdfauthor={Alexander Schmidt},
  pdfsubject={},
  pdfkeywords={},
  colorlinks=true,    
  linkcolor=black,     
  citecolor=black,     
  filecolor=black,      
  urlcolor=black,       
  breaklinks=true,
  bookmarksopen=true,
  bookmarksnumbered=true,
  pdfpagemode=UseOutlines,
  plainpages=false,
  unicode=true
  }
\usepackage[capitalise]{cleveref}
\let\noi\noindent

\newcommand{\ms}{\medskip}

\newcommand{\Gal}{\mathrm{Gal}}

\newcommand{\Hom}{\mathrm{Hom}}

\newcommand{\pr}{\mathit{pr}}
\newcommand{\Nis}{\mathrm{Nis}}

\newcommand{\m}{\mathfrak{m}}

\newcommand{\lang}{\longrightarrow}

\newcommand{\tat}{\mathrm{tt}}
\renewcommand{\O}{\mathcal{O}}

\newcommand{\Ha}{{ \mathcal{H}}}

\newcommand{\Cat}{\mathrm{Cat }}

\newcommand{\et}{\mathrm{et}}

\newcommand{\FEt}{\mathrm{FEt}}
\newcommand{\FT}{\mathrm{FT}}

\newcommand{\pro}{\mathrm{pro\textit{-}}}

\newcommand{\sh}{\mathit{sh}}

\newcommand{\A}{{\mathds A}}

\newcommand{\Sh}{\mathit{Sh}}

\newcommand{\cC}{{\mathscr C}}

\newcommand{\cF}{{\mathscr F}}

\newcommand{\cH}{{\mathscr H}}

\newcommand{\liso}{\mathrel{\hbox{$\longrightarrow$} \kern-2.4ex\lower-1ex\hbox{$\scriptstyle\sim$}\kern1.7ex}}

\newtheoremstyle{alexthm}
  {}
  {}
  {\sl }
  {}
  {\bf}
  {.}
  {.5em}
  {}
\theoremstyle{alexthm}

\newtheorem{theorem}{Theorem}[section]
\newtheorem*{theorem*}{Theorem}
\newtheorem{corollary}[theorem]{Corollary}
\newtheorem{proposition}[theorem]{Proposition}
\newtheorem{lemma}[theorem]{Lemma}
\newtheorem*{lemma*}{Lemma}

\newtheoremstyle{alexdef}
  {}
  {}
  {\rm }
  {}
  {\bf}
  {.}
  {.5em}
  {}
\theoremstyle{alexdef}
\newtheorem*{example*}{Example}

\newtheorem{remark}[theorem]{Remark}

\newtheorem{definition}[theorem]{Definition}

\DeclareMathOperator{\Spec}{\textrm{Spec}}

\newcommand{\sets}{\mathrm{sets}}

\DeclareMathOperator{\ch}{char}
\DeclareMathOperator{\Val}{Val}

\DeclareMathOperator*{\colim}{colim}

 \newcommand{\D}{\Delta^{\!\mathit{op}}}
  \newcommand{\sm}{\mathrm{Sm}/S}
  \newcommand{\shvt}{{\mathit{Shv}_{\mathrm{t}}}(\sm)}

\title{Homotopy invariance of tame homotopy groups of regular schemes}
\author{Alexander Schmidt}
\date{}
\begin{document}
\maketitle
\tableofcontents
The \'{e}tale homotopy groups of schemes as defined by Artin and Mazur \cite{AM} have the disadvantage of being homotopy invariant only in characteristic zero. This and other related problems led to the definition of the \emph{tame} topology which is coarser than the \'{e}tale topology by disallowing wild ramification along the boundary of compactifications, see \cite{HS21}. The objective of this paper is to show that the associated tame homotopy groups are indeed ($\A^1$-)homotopy invariant, at least for regular schemes.

\bigskip
The author thanks the unknown reviewer of this article for his detailed and constructive criticism.

\section{The tame site}

Throughout this article, let $S$ be a (base) scheme and $X$ an $S$-scheme. We recall the definition of the tame site $(X/S)_t$ from \cite{HS21}.

By a valuation on a field we mean a non-archimedean valuation, not necessarily discrete or of finite rank. The trivial valuation is included. If $v$ is a valuation, we denote by $\O_v,\m_v$ and $k(v)$ the valuation ring, its maximal ideal and the residue field. If $v$ is a valuation on $K$, we denote by $\O_v,\m_v$ and $k(v)$ the valuation ring, its maximal ideal and the residue field. By $\O_v^h$ and $\O_v^\sh$ we denote the henselization and strict henselization of $\O_v$ and by $K_v^h$ and $K_v^\sh$ their quotient fields.

Let $v$ be a valuation on $K$ and $w$  an extension of $v$ to a finite separable extension field $L/K$. We call $w/v$ \emph{unramified} if $\O_v \to \O_w$ is \'{e}tale, i.e., $L_w^\sh =K_v^\sh$,  and \emph{tamely ramified} if the field extension $L_w^\sh /K_v^\sh$ is of degree prime to the residue characteristic $p=\ch k(v)$. In this case $L_w^\sh /K_v^\sh$ is automatically Galois with abelian Galois group of order prime to $p$. If $L/K$ is Galois, then $w/v$ is unramified (resp.\ tamely ramified) if and only if the inertia group $T_w (L/K)$ (resp.\ the ramification group $R_w(L/K)$) is trivial. (See \cite{Ray70} and \cite{EP2005}.)

Let $f: X\to S$ be a scheme morphism.
An \emph{$S$-valuation} on $X$ is a valuation $v$ on the residue field $k(x)$ of some point $x\in X$ such that there exists a morphism
$
\varphi: \Spec(\O_v) \to S
$
making the diagram
\[
\begin{tikzcd}
\Spec(k(x)) \arrow[rr] \arrow[d] &  &X \arrow[d,"f"] \\
\Spec(\O_v) \arrow[rr, "\varphi"] &  & S
\end{tikzcd}
\]
commutative (if $S$ is separated, $\varphi$ is unique if it exists).
The set of all $S$-valuations is denoted by $\Val_S X$.
We denote elements of $\Val_S X$ in the form $(x,v)$, $x\in X$, $v\in \Val_S(k(x))$.

\begin{definition}
The \emph{tame site} $(X/S)_t$ consists of the following data:

\smallskip\noindent
The category $\Cat (X/S)_t$ is the category of \'{e}tale morphisms $p:U\to X$.

\smallskip\noindent
A family $(U_i \to U)_{i\in I} $ of morphisms in $\Cat (X/S)_t$ is a covering if it is an \'{e}tale covering and for every point $(u,v)\in \Val_S(U)$ there
exists an index~$i$ and a point $(u_i,v_i)\in \Val_S U_i$ mapping to $(u,v)$ such that $v_i/v$ is (at most) tamely ramified.
\end{definition}

The small \'{e}tale site $X_\et$, the tame site $(X/S)_t$ and the Nisnevich site $X_\Nis$ all have the same underlying category. Every Nisnevich covering is a tame covering and every tame covering is an \'{e}tale covering. Hence there are natural morphisms of sites
\[
X_\et \stackrel{\alpha}{\lang} (X/S)_t \stackrel{\beta}{\lang} X_\Nis,
\]
where $\alpha_*$ and $\beta_*$ are fully faithful.
In particular, being an \'{e}tale sheaf can be viewed as a property of tame sheaves.

\medskip
A \emph{tame point} of $X/S$ is a pair $(\bar{x},\bar{v})$ where $\bar{x}:\Spec k(\bar{x})\to X$ is a morphism from the spectrum of a field to~$X$ and~$\bar{v}$ is an $S$-valuation on $k(\bar{x})$ such that~$k(\bar{x})$ does not admit a nontrivial,  finite, separable extension to which $\bar{v}$ has a tamely ramified extension. In other words, $k(\bar{x})$ is strictly henselian with respect to~$\bar{v}$ and the absolute Galois group of $k(\bar x)$ is a pro-$p$ group, where $p$ is the residue characteristic of $\bar{v}$ (if $p=0$ this means that the group is trivial).

If $\bar v$ is the trivial valuation, we call the point $(\bar{x},\bar{v})$ \emph{trivial point}. Note that (regardless of the residue characteristic) the field $k(\bar x)$ of a trivial point is separably closed.

A tame point $(\bar{x},\bar{v})$ induces a morphism of sites
\[
\bar x: (\Spec k(\bar{x})/\Spec \O_{\bar v})_t \lang (X/S)_t.
\]
By \cite[Lemma 2.5]{HS21}, the inverse image sheaf functor ${\bar x}^*$ is exact for sheaves of sets as well as of abelian groups. Since every tame covering of $\Spec(k(\bar x))$ splits, the global sections functor on $(\Spec k(\bar{x})/\Spec \O_{\bar v})_t$ is exact. Therefore the functor ``stalk at $\bar x$''
\[
 \cF \longmapsto \cF_{\bar x}:=\Gamma (\Spec k(\bar{x}),\bar x^*\cF),
\]
is a topos-theoretical point of $(X/S)_t$. The trivial points correspond to the usual geometric points of the \'{e}tale site (followed by $\alpha: X_\et \to (X/S)_t$).

\medskip
We denote the category of tame sheaves of abelian groups by $\Sh_t(X/S)$.

\begin{lemma} The family of tame points is conservative for the site $(X/S)_t$. In particular,
a tame sheaf of abelian groups $\cF\in \Sh_t(X/S)$ is zero if and only if its stalks at all tame points are zero.
\end{lemma}

\begin{proof}
See  \cite[Lemma 2.10]{HS21}.
\end{proof}

The following theorem lists some of the properties of the tame site proven in \cite{HS21}:

\begin{theorem}\label{mainproperty}
Let $X$ be an $S$-scheme and let $\alpha: X_\et \to (X/S)_t$ be the natural morphism of sites. Then the following hold:
\begin{enumerate}[\rm (i)]
\item (Topological invariance, \cite[Proposition 3.1]{HS21}) If $X'\rightarrow X$ is a universal homeomorphism of $S$-schemes, then the sites $(X/S)_t$ and $(X'/S)_t$ are isomorphic.
\item (Comparison with \'{e}tale cohomology for invertible coefficients, \cite[Proposition 8.1]{HS21}) Let $\cF\in \Sh_\et(X)$ be an abelian sheaf  with $mF=0$ for some $m$ which is invertible on $S$. Then the natural map
\[
H^n_t(X/S,\alpha_*\cF)\cong  H^n_\et(X,\cF)
\]
is an isomorphism for all $n\ge 0$.

\item (Comparison with \'{e}tale cohomology for proper schemes, \cite[Proposition 8.2]{HS21}) Assume that $S$ is quasi-compact and quasi-separated and that $X\to S$ is proper. Assume moreover that every finite subset of $X$ is contained in an affine open. Then, for every sheaf $\cF$ of abelian groups on $(X/S)_t$  the natural map
\[
H^n_t(X/S, \cF)\to H^n_\et(X,\alpha^*\cF)
\]
is an isomorphism for all $n\ge 0$.
\item (Finite subcoverings, \cite[Theorem 4.1]{HS21}) Assume that $S$ is quasi-compact and quasi-separated and that  $X$ is quasi-compact. Then every tame covering of $X$ admits a finite subcovering.
\item (Colimits of sheaves, \cite[Theorem 4.5]{HS21}) Assume that $S$ and $X$ are quasi-compact and quasi-separated and let $(\cF_i)$ be a filtered direct system of abelian sheaves on $(X/S)_t$. Then
\[
\colim_i \, H^n_t(X/S, \cF_i)\cong H^n_t(X/S, \colim_i \, \cF_i)
\]
for all $n\ge 0$.
\item (Limits of schemes, \cite[Theorem 4.6]{HS21}) Let~$(f_i:X_i \to S_i)_{i\in I}$ be a filtered inverse system of scheme morphisms with affine transition morphisms and assume that all $S_i$ and $X_i$ are quasi-compact and quasi-separated. Denote by $X \to S$ its inverse limit. Assume that~$i_0 \in I$ is a final object and let~$\cF_0$ be a sheaf of abelian groups on $(X_{i_0}/S_{i_0})_t$.
 For $i \in I$ denote by~$\cF_i$ its pullback to~$(X_i/S_i)_t$ and by~$\cF$ its pullback to~$(X/S)_t$.
 Then the natural map
 $$
 \colim_{i \in I}\, H^n_t(X_i/S_i,\cF_i) \longrightarrow H^n_t(X/S,\cF)
 $$
 is an isomorphism for all $n \geq 0$.
\end{enumerate}
\end{theorem}

\section{The profinite tame fundamental group}

The \'{e}tale fundamental group of a geometrically pointed, connected scheme $(X,\bar{x})$ was defined in \cite{SGA1} as the automorphism group of the fibre functor
\[
F_{\bar x}: \FEt_X \lang \sets, \ (Y\to X) \longmapsto \pi_0(Y\times_X{\bar{x}}),
\]
where $\FEt_X$ is the Galois category of finite \'{e}tale $X$-schemes. This is a profinite group. To distinguish it from other constructions, we denote it  by $\hat\pi_1^\et(X,\bar{x})$ and call it the \emph{profinite \'{e}tale fundamental group}.

\medskip
A similar  procedure yields the \emph{profinite tame fundamental group} $\hat \pi_1^t(X/S,\bar{x})$ of a connected $S$-scheme $X$:

\begin{definition} Let $X$ be a connected $S$-scheme. A finite \'{e}tale morphisms $f:Y \to X$ is  called \emph{finite tame cover} if for every  $y\in Y$ and $w\in \Val_S k(y)$ the extension $k(y)/k(f(y))$ is at most tamely ramified at $w$. We define the category $\FT_{X/S}$ as the category with objects  the finite tame covers of $X$ and with $S$-morphisms as morphisms.
 \end{definition}

 \begin{remark}
 Let $Y\to X$ be an \'{e}tale \emph{Galois} cover. Since the Galois group acts transitively on the fibres, we have
 \[
 Y\in \FT_{X/S}\ \Longleftrightarrow \ Y \to X \text{ is a covering in } (X/S)_t.
 \]
 \end{remark}
Now let $(\bar x,\bar v)$ be a tame point of $X/S$. Then
\[
F_{(\bar x,\bar v)}: \FT_{X/S}\lang \sets  ,\ (Y\to X) \longmapsto \pi_0(Y\times_X \bar{x})
\]
is a fibre functor on the Galois category $\FT_{X/S}$.
\begin{definition} We denote the automorphism group of the fibre functor $F_{(\bar x,\bar v)}$ by $\hat \pi_1^t(X/S,(\bar{x},\bar{v}))$. It is a profinite group that we call the \emph{profinite tame fundamental group} of $X/S$.
\end{definition}
\medskip
By the general theory of Galois categories we obtain from \cite[TAG 0BN4]{stacks-project}:

\begin{proposition} \label{finite-pi1-sets} Let $X$ be   a connected $S$-scheme and let $(\bar x, \bar v)$ be a tame point of\/ $X/S$. Then the fibre functor induces a natural equivalence  between $\FT_{X/S}$ and the category of finite discrete $\hat \pi_1^t(X/S,(\bar x,\bar v))$-sets.
\end{proposition}

By topological invariance, $\hat \pi_1^t(X/S,(\bar{x},\bar{v}))$ only depends on the reduction $X_\mathrm{red}$ of $X$.  The open subgroups of $\hat \pi_1^t(X/S,(\bar{x},\bar{v}))$ correspond to the connected finite tame covers of $X$ (Galois if and only if the subgroup is normal). If $(\bar y, \bar w)$ is another tame point, then the groups
$\hat \pi_1^t(X/S,(\bar{x},\bar{v}))$ and $\hat \pi_1^t(X/S,(\bar{y},\bar{w}))$ are isomorphic, the isomorphism being unique up to inner automorphisms. Any commutative diagram
\[
\begin{tikzcd}
X\rar\dar & X'\dar\\
S \rar &S'
\end{tikzcd}
\]
induces a homomorphism $\hat \pi_1^t(X/S,(\bar{x},\bar{v})) \to \hat \pi_1^t(X'/S',(\bar{x}',\bar{v}'))$, where $(\bar{x}',\bar{v}')$ is the point of $X'/S'$ induced by $(\bar{x},\bar{v})$.

\medskip
For a geometric point $\bar{x}=(\bar{x}, \mathrm{triv})$ we obtain a natural surjection
\[
\hat \pi_1^\et(X/S,\bar{x}) \twoheadrightarrow \hat \pi_1^t(X/S,\bar{x})
\]
which is an isomorphism if $X\to S$ is proper.

\bigskip
Our objective is to show that the profinite tame fundamental group is homotopy invariant on normal schemes.

\begin{lemma}\label{a1norm}
Let $X$ be a connected, normal scheme, $K=k(X)$, $L/K$ an algebraic field extension and $Y=X_L$ the normalization of $X$ in $L$. Then $\A^1_Y$ is the normalization of $\A^1_X$ in the field extension $k(\A^1_Y)/k(\A^1_X)$.
\end{lemma}

\begin{proof}
We may assume that $X=\Spec(A)$ is affine, hence $Y=\Spec(B)$, with $B=A_L$ the integral closure of $A$ in $L$. Then $B[T] \subset L(T)$ is integral over $A[T]\subset K(T)$. On the other hand, by \cite[V\,\S1 Cor.\,1 to Prop.\,13]{BourbCA}, $B[T]$ is integrally closed. This completes the proof.
\end{proof}

\begin{theorem}[Homotopy invariance of the profinite tame fundamental group of normal schemes] \label{pi1hi} Let $X$ be a normal connected $S$-scheme and $(\bar{y},\bar w)$ a tame point of\/ $\A^1_X/S$ with image $(\bar{x},\bar v)$ in $X$. Then the projection $\pr: \A^1_X\to X$ induces an isomorphism
\[
\pr_*^{X/S}: \hat\pi_1^t(\A^1_X/S,(\bar y,\bar w)) \lang \hat\pi_1^t(X/S,(\bar x,\bar v)).
\]
\end{theorem}
\begin{proof}
Since the fundamental group is independent of the base point, we may assume that $\bar{y}=(\bar{y}, \mathrm{triv})$ is a geometric point lying over the generic fibre $\A^1_\eta$ of $\A^1_X$, where $\eta\in X$ is the generic point.

Since the zero section to $\pr$ is an $S$-morphism, we conclude that $\pr_*^{X/S}$ is surjective. We have a commutative diagram
\[
\begin{tikzcd}
\pi_1^t(\A^1_\eta/\eta,\bar y)\arrow[d,"\pr^{\eta/\eta}_*",two heads]\arrow[r,two heads]&\pi_1^t(\A^1_\eta/S,\bar y)\arrow[d,"\pr^{\eta/S}_*",two heads]\arrow[r,two heads]& \pi_1^t(\A^1_X/S,\bar y)\arrow[d,"\pr_*^{X/S}",two heads]\\
\pi_1^t(\eta/\eta,\bar x)\arrow[r,two heads]&\pi_1^t(\eta/S,\bar x)\arrow[r,two heads]&\pi_1^t(X/S,\bar x).
\end{tikzcd}
\]
By definition, every finite separable extension of $k(\eta)$ is tame with respect to $k(\eta)$, hence $\pi_1^t(\eta/\eta,\bar x)\cong \Gal(\overline{k(\eta)}/k(\eta))$.
The fundamental exact sequence for the \'{e}tale fundamental groups immediately implies the exactness of the analogous sequence
\[
1 \longrightarrow \pi_1^t(\A^1_{\bar \eta}/{\bar \eta},\bar y) \longrightarrow \pi_1^t(\A^1_\eta/\eta,\bar y) \stackrel{\pr^{\eta/\eta}_*}{\longrightarrow} \pi_1^t(\eta/\eta,\bar x) \longrightarrow 1.
\]
Since the affine line over a separably closed field has no nontrivial tame covers (by the Hurwitz formula), we see that $\pi_1^t(\A^1_{\bar \eta}/{\bar \eta},\bar y)=1$. Hence $\pr^{\eta/\eta}_*$ is an isomorphism. In order to show that $\pr_*^{X/S}$ is an isomorphism, it suffices to show that every connected tame (with respect to $S$) finite cover $Z\to \A^1_X$ is the base change $\A^1_Y \to \A^1_X$ of a finite tame (with respect to $S$) cover $Y\to X$.
Let $K$ be the function field of $X$, hence $\eta=\Spec K$.
We put $Z_\eta=Y\times_X \eta$ and obtain the following commutative diagram
\[
\begin{tikzcd}
  \Spec (k(Z)) \dar \arrow[hook, r]&Z_\eta \dar \arrow[hook,r]& Z\dar\\
  \Spec(k(\A^1_X))\arrow[hook,r]&\A^1_\eta \arrow[hook,r]&\A^1_X.
\end{tikzcd}
\]
Since tame covers are stable under base change, $Z_\eta \to \A^1_\eta$ is tame with respect to $S$, in particular, tame with respect to $\eta$. Since $\pr^{\eta/\eta}_*$ is an isomorphism, $Z_\eta\to \A^1_\eta$ comes by base change from $\eta$, i.e., there is a finite separable field extension $L/K$ such that $Z_\eta\cong \A^1_L$.

As $Z \to \A^1_X$ is an \'{e}tale cover, $Z$ is the normalization of $\A^1_X$ in $k(Z)=k(Z_\eta)=k(\A^1_L)$.
By \cref{a1norm}, the normalization of $\A^1_X$ in $k(\A^1_L)$ is equal $\A^1_Y$, where $Y$ is the normalization of $X$ in $L$. We conclude that $Z\cong\A^1_Y$. It remains to show that the (a priori only integral) morphism $Y\to X$ is a finite tame cover with respect to $S$. This follows since we recover $Y\to X$ as the base change of $Z\to \A^1_X$ along the zero section.
\end{proof}

\section{Comparison with the curve tame fundamental group}

In good situations, there is another notion of tameness which can be formulated without valuation theory only in terms of algebraic geometry. We speak of curve tameness, which was introduced in \cite{KeSch10}. We want to compare it with the notion of tameness of this article.
Unfortunately, the statement of \cite[Proposition 5.2]{HS21} is wrong in that the assumption ``$S=\Spec k$ for a field $k$ and $X/S$ is separated and of finite type'' is missing there. Therefore we take up the subject again and prove a corrected and slightly generalized statement in \cref{pi1ct=pi1t} below.

\ms\noi
Let  $S$ be a scheme which is integral, pure-dimensional (i.e., $\dim  S= \dim \mathcal{O}_{S,s}$ for every closed point $s\in S$), separated and excellent. Furthermore, we let $X\to S$ be separated and of  finite type.
For integral $X$ we put
\[
\dim_S X:= \mathrm{deg.tr.}(k(X)/k(T)) + \dim_{\textrm{Krull}} T,
\]
where $T$ is the closure of the image of $X$ in $S$. If the image of $X$ in $S$ contains a closed point of $T$,  then $\dim_S X=\dim_{\textrm{Krull}} X$ by \cite[5.6.5]{EGAIV.4}.
We call $X$ an \emph{$S$-curve} if $\dim_S X=1$. For a regular connected $S$-curve $C$, let $\bar C$ be the unique regular compactification of $C$ over $S$. Then the finite tame Galois covers of $C$ are exactly those \'{e}tale Galois covers which are tamely ramified along $\bar C \smallsetminus C$ in the classical sense (e.g.\ \cite{GM71}).
We recall the notion of curve-tameness from \cite{KeSch10}:

\begin{definition}
Let $Y\to X$ be a finite \'{e}tale cover of separated $S$-schemes of finite type. We say that $Y\to X$ is {\em curve-tame} if for any morphism $C\to X$ with $C$ a regular $S$-curve, the base change $Y\times_X  C \to  C$ is tamely ramified along $\bar C\smallsetminus C$.
\end{definition}

Now let $X$ be connected and choose a geometric point $\bar x$. We obtain the \emph{profinite curve-tame fundamental group} $\hat \pi_1^{ct}(X/S,\bar x)$, which is the quotient of $\hat \pi_1^\et(X,\bar x)$ that classifies finite curve-tame covers. This curve-tame fundamental group was considered in \cite{Sch-singhom}, \cite{KeSch09}, \cite{GS16}.

\bigskip
For a general base scheme $S$, our comparison result depends on local uniformization:
\begin{definition}
\label{locunif}Let~$S$ be a noetherian scheme. We say that \emph{local uniformization} holds over $S$, if for any integral scheme $X$ of finite type over $S$ with function field $k(X)$ and any $S$-valuation $v$ of $k(X)$ with center on $X$ there exists a connected, regular scheme $X'$ and a morphism $X'\to X$ inducing an isomorphism $k(X')\cong k(X)$ and such $v$ has center on $X'$.
\end{definition}

\begin{remark}
Local uniformization is a weaker property than resolution of singularities. Over fields of characteristic zero it was proven by Zariski \cite{Zar-loc}. Over fields of characteristic $p>0$, Temkin proved an inseparable variant in \cite{Tem-LU}, which was used in an essential way in \cite{quasi-purity} (which we use below).
\end{remark}

\begin{proposition}\label{pi1ct=pi1t} Let $S$ be an integral, pure-dimensional, separated and excellent scheme, $X\to S$ separated and of  finite type and $\bar{x}$ a geometric point of $X$. Assume that $S$ is the spectrum of a field or that local uniformization hold over $S$. Then the profinite curve-tame fundamental group coincides with the profinite tame fundamental group:
\[
\hat \pi_1^{t}(X/S,\bar x) \cong \hat\pi_1^{ct}(X/S,\bar x).
\]
\end{proposition}

\begin{proof} If  $\ch(S)=0$, the statement of the proposition is true for trivial reasons, since both sides coincide with the profinite \'{e}tale fundamental group. Hence we can assume that $S$ has at least one point of positive residue characteristic.
We have to show that a finite \'{e}tale Galois cover of $X$ is tame if and only if it is curve-tame. Both notions coincide for regular $S$-curves and are stable under base change. Hence finite tame Galois covers are curve-tame and it remains to show the converse.

So let $f: Y\to X$ be a finite \'{e}tale Galois cover, $y\in Y$ a point, $x=f(y)$ and $w\in \Val_S k(y)$ such that $w$ is wildly ramified in $k(y)/k(x)$. We have to find an $S$-curve $C$ and a morphism $\varphi: C\to X$ such that $C\times_XY\to C$ is wildly ramified at some point in $\bar C \smallsetminus C$.

Let $X'$ be the closure of $\{x\}$ in $X$ with reduced scheme structure and $Y'=X'\times_XY$. In order to find $\varphi: C\to X$, we can replace $X$ by $X'$ and $Y$ by the connected component of $y$ in $Y'$, i.e., we may assume that $x$ and $y$ are the generic points of the integral schemes $X$ and $Y$.
Replacing $X$ by its normalization $\widetilde X$ and $Y$ by its base change along $\widetilde X\to X$, we may assume that $X$ and $Y$ are normal.

\medskip
We first assume that $\Gal(Y/X)$ is cyclic of order $p:=\ch k(w)$. Then, since $w$ is ramified, $k(y)/k(x)$ is also of degree $p$.

If $S=\Spec (k)$ for a field $k$ (necessarily of characteristic $p$), then, by quasi-purity of the branch locus \cite{quasi-purity}, we find a geometric discrete rank-one $S$-valuation $v$ on $k(y)$ that is ramified in $k(y)/k(x)$ and specializes to $w$.

In the general case, assume that local uniformization holds over $S$. Let $\bar X$ be a proper (over $S$) normal compactification of $X$ and  let $z \in \bar{X}$ be the center of~$w$.   We choose a connected, regular scheme $X'$ with a morphism $X'\to \bar X$ inducing an isomorphism $k(X')\cong k(X)$ and such $w$ has center $z'$ on $X'$ and let $Y'$ be the normalization of $X'$ in $k(Y)$. Then $Y'/X'$ ramifies over $z'$ and by the Zariski-Nagata theorem on the purity of the branch locus we find a ramified divisor on $X'$ which contains $z'$. This gives us  a geometric discrete rank-one $S$-valuation $v$ on $k(Y)$ that is ramified in $k(Y)/k(X)$.

In both cases,  by \cite[\S 8,\, Theorem\,3.26 and Exercise 3.14]{Liu},
we can find a normal compactification (over $S$) $\bar X$ of $X$ such that the center $D$ of $v$ on $\bar X$ is of codimension one.  Applying  the Key Lemma 2.4 of \cite{KeSch10} to the local ring of some closed point of $\bar X$ contained in $D$ and with residue characteristic $p$, we obtain a morphism $\varphi: C\to X$ with the property that $C\times_XY\to C$ is ramified (hence wildly ramified) at this point.

\smallskip\noindent
It remains to reduce the general case to the cyclic-order-$p$-case. For this let
\[
A\subset R_w(k(Y)/k(X)) \subset \Gal(k(Y)/k(X))
\]
be a subgroup of order $p$ (here $R_w$ denotes the ramification group) and let $X'=Y_A$. Then $Y/X'$ is of degree $p$ and wildly ramified at $w$. By the first part of the proof, we find a morphism $\varphi: C \to X'$ with the required property. Then the composite of $\varphi$ with the projection $X'\to X$ yields what we need.
\end{proof}

\section{Locally constant tame sheaves}

For an abelian group $A$ we denote by $\underline{A}$ the Zariski sheafification of the constant presheaf associated with $A$ on the category $X_\et$. It is an \'{e}tale sheaf, in particular a tame sheaf.

\begin{definition}
A sheaf of abelian groups $\cF\in \Sh_t(X/S)$ is \emph{locally constant} if there is a tame covering $(U_i \to X)$ such that the restriction $\cF|_{U_i}$ is isomorphic to a constant sheaf $\underline{A_i}$ on $U_i$ for all $i$.
\end{definition}
Since constant sheaves are \'{e}tale sheaves, the same is true for locally constant sheaves by \cref{etale sheaf} below.

\begin{lemma}\label{etale sheaf}
Let $\cF\in \Sh_t(X/S)$ and assume there is a tame covering $(U_i \to X)$ such that the restrictions of $\cF$ to all $U_i$ are \'{e}tale sheaves. Then $\cF$ is an \'{e}tale sheaf.
\end{lemma}

\begin{proof}
Let $U\in X_\et$ and let $(U_j \to U)_j$ be an \'{e}tale covering. Then $(U_i \times_X U \to U)_i$ is a tame covering, $(U_i \times_X U_j \to U_i\times_X U)_j$ is an \'{e}tale covering for all $i$ and $(U_i\times_X U_j\to U_j)_i$ is a tame covering for all $j$. Giving a section of $\cF$ over $U$ is therefore equivalent to give compatible sections of $\cF$ over all $U_i \times_X U$, which is equivalent to give compatible sections of $\cF$ over all $U_{i}\times_X U_j$ which is equivalent to give compatible sections of $\cF$ on all $U_j$.
\end{proof}

\begin{remark}
We obtain the inclusions:
\[
\{\text{l.c.~tame sheaves}\} \subset \{\text{l.c.~\'{e}tale sheaves}\} \subset \{\text{\'{e}tale sheaves}\} \subset \{\text{tame sheaves}\}.
\]
\end{remark}

\begin{proposition} \label{finite_group-scheme} Let $X$ be a connected $S$-scheme. A locally constant tame sheaf $\cF$ of abelian groups on $(X/S)_t$ has finite stalks if and only if it is represented by an abelian group scheme $G\to X$ that is a finite tame cover of $X$.
\end{proposition}

\begin{proof} The if part is obvious. Let $\cF$ be a locally constant tame sheaf with finite stalks. By \cref{etale sheaf}, $\cF$ is an \'{e}tale sheaf. Hence, by \cite[V,\,Proposition 1.1]{Milne80}, $\cF$ is represented by an abelian finite  \'{e}tale group scheme $f: G\to X$. Since $G$ is a $G$-torsor, it is trivialized by itself, i.e., $\cF|_G$ is a constant sheaf. It remains to show that $f: G \to X$ is a tame cover, i.e., that for every $y\in G$ and $w\in \Val_S k(y)$ the extension $k(y)/k(f(y))$ is at most tamely ramified at $w$.

Since $\cF$ is locally constant, there exists a tame covering $(U_i \to X)_i$ such that $G\times_X U_i \to U_i$ is constant for all $i$. Now let $y\in G$ and $w\in \Val_S k(y)$, $x=f(y)\in X$ and $v=w|_{k(x)}$. We can find an index $i$, $u_i \in U_i$, $v_i \in \Val_S  k(u_i)$ such that $v_i/v$ is at most tamely ramified. Now we pick a point $y_i \in G\times_X U_i$ and $w_i\in \Val_S k(y_i)$ mapping to $(u_i,v_i)$ and $(y,w)$. Since $G\times_X U_i \to U_i$ is constant, we obtain an inclusion $k(y) \subset k(u_i)$ such that $v_i|_{k(y)}  =w$. Hence $w/v$ is dominated by $v_i/v$ and hence tame.
\end{proof}

Combining \cref{finite_group-scheme} with the equivalence of \cref{finite-pi1-sets}, we obtain

\begin{proposition} Let $X$ be a connected $S$-scheme and let $(\bar x, \bar v)$ be a tame point of\/ $X/S$. Then there is a natural equivalence  between the category of finite discrete $\hat \pi_1^t(X/S,(\bar x,\bar v))$-modules and the category of sheaves on $(X/S)_t$ which are locally constant and have finite stalks.
\end{proposition}

\section{\'{E}tale and tame homotopy groups} \label{hgsect}
\medskip
In \cite{AM}, Artin and Mazur defined higher \'{e}tale homotopy groups. The construction is the following:

A category $C$ admitting finite fibre products is called distributive if it has an initial object $\varnothing$, and if the following condition holds: For every family of objects $Y_i$, $i\in I$, such that the coproduct $\coprod_{i\in I} Y_i$ exists in $C$, any family of morphisms $Y_i\to S$ and for any morphism $X\to S$, the canonical morphism $\coprod_{i\in I} X \times_S Y_i \to X \times_S (\coprod_{i\in I} Y_i)$ is an isomorphism.

An object  $X$ of a distributive category $C$ is called connected if it is not the initial object $\varnothing$ and has no non-trivial copro\-duct decomposition. The category $C$ is called locally connected if every object has a copro\-duct decomposition into connected objects. This decomposition  is essentially unique then, so we can speak about the set of connected components of an object of $C$.
The rule associating to an object its set of connected components is then a functor, denoted by
\[
\Pi: C \lang \sets.
\]
Assume in addition that $C$ is a pointed site.
We consider the category $\textrm{HR}(C)$, whose objects are pointed hypercoverings of $C$ and whose maps are homotopy classes of morphisms. The category $\textrm{HR}(C)$ is left filtering \cite[Corollary 8.13]{AM}.
By applying the connected components functor $\Pi$ to $\textrm{HR}(C)$,  we obtain a pro-object
\[
\Pi C =(\pi(K_\bullet))_{K_\bullet}
\]
in the homotopy category of pointed simplicial sets. The homotopy pro-groups of $C$ are defined by
\[
\pi_q(C)= \pi_q(\Pi C).
\]
For a locally noetherian $S$-scheme $X$, the above applies to the sites $X_\et$ and $(X/S)_t$ and one puts for a geometric point $\bar{x}$ of $X$, resp.\ a tame point $\bar x=(\bar x, \bar v)$ of $X/S$:
\[
\pi_q^\et(X,\bar{x})=\pi_q(X_\et, \bar x), \quad \pi_q^t(X/S,(\bar{x},\bar{v}))=\pi_q((X/S)_t, (\bar x,\bar v)).
\]
The following \cref{tame-homotop} verifies properties of the tame homotopy groups that were proven in \cite{AM} and \cite{Frie82} in the \'{e}tale case.
\begin{theorem} \label{tame-homotop} Let $X$ be a locally noetherian $S$-scheme and $(\bar x,\bar v)$  a tame point of $X/S$. Then the following holds.
\begin{enumerate}[\rm (i)]
 \item $ \pi_0^t(X/S,(\bar{x},\bar v))$ is the set $\pi_0(X)$ of connected components of $X$, pointed in the component of $\bar{x}$.
 \item If $X_0$ is the connected component of $\bar x$, then the natural homomorphism
 \[
 \pi_q((X_0/S)_t, (\bar x,\bar v)) \lang \pi_q((X/S)_t, (\bar x, \bar v))
 \]
 are isomorphisms of pro-groups for all $q\ge 1$.
  \item There is a natural homomorphism of pro-groups
  \[
  \pi_1^t(X/S,(\bar{x},\bar v)) \lang   \hat \pi_1^t(X/S,(\bar{x},\bar v))
  \]
  which is universal for homomorphisms of $\pi_1^t(X/S,(\bar{x},\bar v))$ to profinite groups. In other words, $\hat \pi_1^t(X/S,(\bar{x},\bar v))$ is the profinite completion of $\pi_1^t(X/S,(\bar{x},\bar v))$.
  \item  Assume that $X$ is noetherian and geometrically unibranch (e.g., normal) and that $S$ is quasi-compact and quasi-separated. Then the pro-groups
  \[
   \pi_q^t(X/S,(\bar{x},\bar v))
  \]
  are profinite for all $q\ge 1$.
\end{enumerate}
\end{theorem}
\begin{proof} (i) We show that for any tame hypercovering $K_\bullet \to X$ the natural map $\pi_0(\pi(K_\bullet))\to \pi_0(X)$ is a bijection. This shows that $ \pi_0^t(X/S,(\bar{x},\bar v))$ is the constant pro-object associated with the set $\pi_0(X)$ (pointed in the component of $\bar{x}$).  By Yoneda, it suffices to show that for any set $S$ the induced map
\[
S^{\pi_0(\pi(K_\bullet))} \lang S^{\pi_0(X)}
\]
is bijective. Let $\underline{S}$ be the constant tame sheaf of sets associated with $S$. Then $\Gamma(Y,\underline{S})=S^{\pi_0(Y)}$ for every $Y\in X_\et$. Similarly, we obtain
$\Gamma(Y_\bullet,\underline{S})=S^{\pi_0(\pi(Y_\bullet))}$ for every  simplicial object $Y_\bullet$ of $X_\et$. For a hypercovering $K_\bullet \to X$, the exact sequence
\[
\Gamma(K_\bullet, \underline{S}) \to \Gamma(K_0, \underline{S}) \rightrightarrows \Gamma(K_1, \underline{S}),
\]
and the sheaf property of $\underline{S}$ show $\Gamma(K_\bullet, \underline{S})= \Gamma(X,\underline{S})$, hence $S^{\pi_0(\pi(K_\bullet))} = S^{\pi_0(X)}$.

\noindent
(ii) This follows since the functor $\pi$ commutes with disjoint unions and because the homotopy groups of a pointed simplicial set only depend on the connected component of the base point.

\noindent
(iii) Recall that the category of profinite groups is the pro-category of the category of finite groups. Moreover, recall that any pro-object $X=(X_i)_{i\in I}$ of a category $\cC$ defines the covariant functor
\[
\Hom(X,-): \cC \lang \sets, \ Y \longmapsto \colim_i \Hom(X_i,Y)
\]
and that this construction identifies the category $\pro\cC$ with the full subcategory of pro-representable covariant functors from $\cC$ to $\sets$. For a (discrete) group $G$ (considered as a constant group scheme over $X$), we call a $G$-torsor $T\to X$ tame, if $T\to X$ is a tame covering.  Let $\pi^1_t(X/S,G,(\bar{x},\bar v))$ denote the set of isomorphism classes of pointed (over $(\bar x,\bar v)$) tame $G$-torsors. The profinite tame fundamental group $\hat\pi_1^t(X/S,(\bar{x},\bar{v}))$ represents the functor
\[
\textrm{finite groups} \lang \sets, \ G \longmapsto \pi^1_t(X/S,G,(\bar{x},\bar{v}) ).
\]
On the other hand, by \cite[Corollary 10.7]{AM}, the pro-group $\pi_1^t(X/S,(\bar{x},\bar{v}))$ represents the functor
\[
\textrm{groups} \lang \sets, \ G \longmapsto \pi^1_t(X/S,G,(\bar{x},\bar{v})).
\]
From this the assertion of (iii) follows formally.

\noindent (iv) Since $X$ is noetherian, by \cref{mainproperty}(iv),  the noetherian hypercoverings (i.e., those $K_\bullet \to X$ with $K_q$ noetherian for all $q$) are cofinal among all hypercoverings of $X$. Therefore it suffices to show that for every noetherian tame hypercovering $K_\bullet \to X$ all homotopy groups of the simplicial set $\pi(K_\bullet)$ are finite. Since tame hypercoverings are \'{e}tale hypercoverings, this follows from \cite[Theorem 11.2]{AM}.
\end{proof}

\section{Homotopy invariance of tame homotopy groups}
\begin{definition}
\label{ros}Let~$S$ be a noetherian scheme.
 We say that \emph{resolution of singularities holds over~$S$} if for any reduced scheme~$X$ of finite type over~$S$ there is a locally projective birational morphism $X' \to X$ such that~$X'$ is regular and $X' \to X$ is an isomorphism over the regular locus of~$X$.  (By \cite[IV, 7.9.5]{EGAIV.2}, this particularly implies that $S$ is quasi-excellent.)
 \end{definition}

 The following theorem was proven in \cite[Theorem 15.2]{HS21}:

 \begin{theorem}[Homotopy invariance of tame cohomology] \label{homotopyinvariancecoh}
 Let~$S$ be an affine noetherian scheme of characteristic $p > 0$ and $X$ a regular scheme which is essentially of finite type over~$S$.
 Assume that resolution of singularities holds over~$S$.
 Then for every locally constant torsion sheaf $F\in \Sh_t(X/S)$ the natural map
 \[
   H^q_t(X/S,F) \lang H^q_t(\A^1_X/S,\pr^*F),
 \]
where $\pr: \A^1_X \to X$ is the natural projection, is an isomorphism for all $q\ge 0$.
\end{theorem}

Together with our results on the fundamental group, this implies

 \begin{theorem}[Homotopy invariance of tame homotopy groups for regular schemes] \label{homotopyinvariance}
 Let~$S$ be an affine noetherian scheme of characteristic $p > 0$ and $X$ a regular scheme which is essentially of finite type over~$S$.
 Assume that resolution of singularities holds over~$S$.
 Then for every tame point $(\bar{y},\bar{w})$ of\/ $\A^1_X/S$ with image $(\bar{x},\bar{v})$ in $X/S$, the map
  \[
\pi^q_t(\A^1_X/S,(\bar{y},\bar{w})) \lang  \pi^q_t(X/S,(\bar{x},\bar{v})),
 \]
induced by the projection $\pr: \A^1_X \to X$, is an isomorphism for all $q\ge 0$.
\end{theorem}

\begin{proof}
The statement for $q=0$ follows from \cref{tame-homotop}(i) together with the fact that $\A^1_X$ is connected if $X$ is. By assumption, $X$ is normal, hence the same is true for $\A^1_X$. By \cref{tame-homotop}(iv), $\pi^q_t(\A^1_X/S,(\bar{y},\bar{w}))$ and $ \pi^q_t(X/S,(\bar{x},\bar{v}))$ are profinite for all $q\ge 1$. In particular, \cref{pi1hi} and \cref{tame-homotop}(iii) show the statement for $q=1$.

By \cite[Theorem 4.3]{AM} and \cref{homotopyinvariancecoh}, we conclude that the natural map from the profinite completion of the homotopy type of $(\A^1_X/S)_t$ to the profinite completion of the homotopy type of $(X/S)_t$ is a weak equivalence. Finally, since the tame homotopy groups of $X/S$ and $\A^1_X/S$ are profinite, the  homotopy types of $(\A^1_X/S)_t$ and $(X/S)_t$ are already profinite.
We conclude $\pi^q_t(\A^1_X/S,(\bar{y},\bar{w})) \cong \pi^q_t(X/S,(\bar{x},\bar{v}))$ for all $q$.
\end{proof}

\begin{remark}
If $S$ is pure of characteristic zero and $X$ is an $S$-scheme, then the sites $X_\et$ and $(X/S)_t$ coincide and the statements of \cref{homotopyinvariance,homotopyinvariancecoh} are well known.
\end{remark}

\section{\texorpdfstring{Tame homotopy groups as a functor on $\cH_{\A^1}(\mathrm{Sm}/k)$}{Tame homotopy groups as a functor on HA¹(Sm/k)}}

In this section we will prove that (under the assumption of resolution of singularities) the notion of tame homotopy groups extends  to the $\A^1$-homotopy category $\cH_{\A^1}(\mathrm{Sm}/k)$ of smooth schemes over a field $k$ (as defined by Morel and Voevodsky \cite{MoVo}). This is not the case for the \'{e}tale homotopy groups if $\ch k>0$ (the affine line has a huge fundamental group then).

\bigskip
In fact, we will show a little more. Let~$S$ be a connected affine noetherian regular scheme of characteristic $p > 0$ and assume that resolution of singularities holds over~$S$. Recall that the definition of the  $\A^1$-homotopy category $\cH_{\A^1}(\mathrm{Sm}/S)$ of smooth schemes over $S$ is based on the Nisnevich topology and we will use the more precise notation
$\cH_{\A^1,\Nis}(\mathrm{Sm}/S)$ from now on. The same construction works with the tame and the \'{e}tale topology yielding the categories $\cH_{\A^1,t}(\mathrm{Sm}/S)$ and $\cH_{\A^1, \et}(\mathrm{Sm}/S)$. There are natural functors
\[
\mathrm{Sm}/S \to \cH_{\A^1, \Nis}(\mathrm{Sm}/S)\to  \cH_{\A^1,t}(\mathrm{Sm}/S)\to \cH_{\A^1, \et}(\mathrm{Sm}/S),
\]
where the first functor is induced by the Yoneda embedding and the functors between the various  $\A^1$-homotopy categories are induced by sheafification.

\medskip
Recall from \cref{hgsect} that the tame homotopy groups of geometrically pointed smooth schemes over $S$ are defined as the composition of the functor `pointed tame homotopy type'
\[
\Pi_\bullet^t: (\mathrm{Sm}/S)_\bullet \to \pro\Ha_\bullet
\]
where $\pro \Ha_\bullet$ is the pro-category of the homotopy category of pointed simplicial sets with the usual  homotopy group functor
\[
\pro\Ha_\bullet \stackrel{\pi_*}{\lang} \pro \mathrm{groups}
\]
(for $*=0$ replace $\pro \mathrm{groups}$ by pointed $\pro\sets$).

\medskip
Let $X=\{X_i\}$ be in $\pro\Ha$.  The various coskeletons $\mathrm{cosk}_nX_i$ form a pro-object $X^\natural$ indexed by pairs $(i,n)$.  We have a natural map $X\to X^\natural$ and $X^\natural \to X^{\natural\natural}$ is an isomorphism. We call a map $f: X \to Y$ in $\pro\Ha$ a weak equivalence ($\natural$-isomorphism in \cite{AM}) if $f^\natural: X^\natural \to Y^\natural$
is an isomorphism.
Let
\[
(\pro\Ha)_w
\]
be the full subcategory in $\pro\Ha$ consisting of objects isomorphic to $X^\natural$ for some~$X$. Then $(\pro\Ha)_w$ is the localization of $\pro\Ha$ with respect to the class of weak equivalences and $\natural: \pro\Ha \to (\pro\Ha)_w$ is the localization functor. The  homotopy pro-groups functors $\pi_*$ factor as
\[
\pi_*:  \pro\Ha_\bullet \to (\pro\Ha_\bullet)_w \to \pro\mathrm{groups}.
\]
Given all this, the goal of this section is to prove the following

\begin{theorem} \label{a1factor} Let~$S$ be a connected affine noetherian regular scheme of characteristic $p > 0$ and assume that resolution of singularities holds over~$S$. Then there exists a dashed vertical arrow making the diagram
\[
\begin{tikzcd}
\mathrm{Sm}/S \rar\dar{\Pi^t}& \cH_{\A^1, \Nis}(\mathrm{Sm}/S) \rar&  \cH_{\A^1,t}(\mathrm{Sm}/S)\arrow[d,dashed]\\
 \pro\Ha\arrow[rr, "\natural"]&& (\pro\Ha)_w
\end{tikzcd}
\]
commutative.
\end{theorem}
Since the structure map of a vector bundle is an isomorphism in $\cH_{\A^1, \Nis}(\mathrm{Sm}/S)$, we particularly obtain
\begin{corollary}
Every vector bundle over a smooth $S$-scheme $X$ has the same tame homotopy groups as $X$.
\end{corollary}
\begin{remark}
The profinite completion functor from $ \pro\Ha$ to the homotopy category $\hat\Ha$ of profinite spaces (as defined in \cite{MorelProf}, \cite{QuiProf}) factors through $(\pro\Ha)_w$. Therefore \cref{a1factor} also provides a functor $\cH_{\A^1,t}(\mathrm{Sm}/S) \to \hat\Ha$.
\end{remark}
Recall that the construction of the $\A^1$-homotopy category after Morel/Voevodsky consists of two steps. The first one is to localize the category of simplicial sheaves by simplicial weak equivalences and the second step consists of inverting the class of weak $\A^1$-equivalences.

\medskip
The proof of \cref{a1factor} will occupy the rest of this section. It is essentially equivalent to the proof of \cite[Theorem 8.2]{Sch-Mot}, where the analogous result of \cref{a1factor} was proven for the \'{e}tale topology after completion away from all residue characteristics. We will therefore be sketchy in the sense that we will use the results of \cite{Sch-Mot} as long their proofs in the tame case are identical to those in the \'{e}tale case.

\medskip
For the rest of this section let $S$ be a locally noetherian scheme. We  consider the category $$\D\shvt$$ of simplicial tame sheaves
(of sets) on $\sm$.  By a point we will always mean a tame point. A map of simplicial
sheaves $f: F\to G$ is called a simplicial weak equivalence if for every point $x$ the map
$F_x \to G_x$ is a weak equivalence of simplicial sets. $f$ is called a
cofibration (resp.\ trivial cofibration) if it is injective (resp.\ injective and a weak equivalence). Fibrations are maps
satisfying the right lifting property with respect to trivial cofibrations.  The category of
simplicial sheaves together with these three classes of morphisms is a simplicial closed
model category (\cite[1.1.4]{MoVo}) and we denote the associated homotopy category by $\Ha_{s,t}(\sm)$.

 A map of simplicial sheaves $F\to
G$ is called a  local fibration (resp.\ trivial local fibration) if for every point $x$ the map $F_x \to G_x$ is a
fibration (resp.\ a fibration and a weak equivalence).  A local fibration has the right lifting
property after a tame refinement. Kan-simplicial sets considered as constant simplicial
sheaves are locally fibrant.

The proof of the following proposition is word by word the same as the proof of \cite[Proposition 5.1]{Sch-Mot} (which works for any subcanonical topology).

\begin{proposition} The category $\shvt$ is locally connected. A connected scheme $X\in \sm$ represents a connected sheaf.
\end{proposition}

The rule associating to a sheaf its set of connected components defines the connected
component functor
\[
\Pi: \shvt \lang \sets.
\]
This functor naturally extends to simplicial sheaves (taking values in simplicial sets).
Furthermore, simplicial homotopies between maps of simplicial sheaves carry over to homotopies
between maps of  simplicial sets.

\medskip\medskip
For simplicial sheaves $\mathcal{X}$, $\mathcal{Y}$ we denote by $\pi(\mathcal{X},\mathcal{Y})$ the
quotient of $\Hom(\mathcal{X},\mathcal{Y})=S_0(\mathcal{X},\mathcal{Y})$ with respect to the equivalence
relation generated by simplicial homotopies, i.e., the set of connected components of the
simplicial function object $S(\mathcal{X},\mathcal{Y})$, and call it the set of {\em simplicial
homotopy classes of morphisms} from $\mathcal{X}$ to $\mathcal{Y}$. The
simplicial homotopy relation is compatible with composition and thus one gets a category $\pi
\Delta^{op}\shvt$ with objects the simplicial sheaves and morphisms the simplicial homotopy
classes of morphisms.

For a simplicial sheaf $\mathcal{X}$ we denote by $\pi \mathit{Triv}/\mathcal{X}$ the category whose objects are the trivial local fibrations to $\mathcal{X}$
and whose morphisms are  commutative triangles in $\pi \Delta^{op}\shvt$. This
category is filtering and essentially small  by \cite{MoVo}, 2.1.12.

We consider a sheaf $F$ as a simplicial sheaf with $F$ in each degree and all face and degeneracy morphisms  the identity of $F$.  For a scheme $X\in \sm$ we denote by $h_X$ the tame sheaf that it represents. A hypercovering $U_{\cdot}$ of $X$
is a hypercovering in the small tame site  $(X/S)_t$.

\begin{lemma}\label{hypercovpitriv}
Let $X$ be a smooth $S$-scheme and let $U_{\cdot}$ be a  hypercovering of $X$. Then the projection
\[
h_{U_{\cdot}} \lang h_X
\]
is a trivial local fibration in $\D\shvt$.

The trivial local fibrations of the form $h_{U_{\cdot}} \to h_X$ are cofinal in $\pi \mathit{Triv}/h_{X}$.
\end{lemma}

\begin{proof}
The similar assertion for \'{e}tale hypercoverings is \cite[Lemmas 4.2,4.3]{Sch-Mot} and the proof is word-by-word the same in the tame case.
\end{proof}

For a scheme $X$ essentially of finite type over $S$ we denote its tame homotopy type, i.e., the connected component functor from tame hypercoverings of $X$ to $\Ha$ by
\[
X_\tat \in \pro\Ha.
\]
By \cref{hypercovpitriv}, the next definition extends via Yoneda the functor tame homotopy type from $\sm$ to $\D\shvt$:
\begin{definition}
For a simplicial sheaf $\mathcal{X} \in \D\shvt$ the tame homotopy type $\mathcal{X}_\tat\in \pro\Ha$ is defined as the connected component  functor
\[
\Pi: \pi \mathrm{Triv}/\mathcal{X}  \lang \Ha .
\]
\end{definition}

With this definition, the same proof as that of \cite[Theorem 6.4]{Sch-Mot} shows factorization through the simplicial homotopy category.
\begin{proposition}
The functor tame homotopy type
\[
\tat: \D\shvt \lang \pro\Ha
\]
factors through the localization $\D\shvt \to \Ha_{s,t}(\sm)$.
\end{proposition}

\cref{factor} completes the proof of \cref{a1factor}.

\begin{theorem}\label{factor}  Let~$S$ be a connected affine noetherian regular scheme of characteristic $p > 0$ and assume that resolution of singularities holds over~$S$. Then the composite
\[
\natural \circ  \tat: \Ha_{s,t}(\sm) \lang (\pro\Ha)_w
\]
factors through the tame $\A^1$-homotopy category $\Ha_{\A^1,t}(\sm)$.
\end{theorem}

\begin{proof}
For any smooth scheme $U$ over $S$ the
projection $\A^1_U \to U$ induces a weak equivalence $${(\A^1_U)}_\tat \to
U_\tat$$ in $\pro\Ha$. Indeed, in order to show this, we may assume that $U$ is connected. Then (after choosing suitable base points), the projection induces isomorphisms on all tame homotopy groups by \cref{homotopyinvariance} and is hence a weak equivalence by \cite[Corollary 4.2]{AM}.

From this, the assertion of the theorem follows via a formal process which is described in detail in the proof of \cite[Theorem 8.2]{Sch-Mot}.
\end{proof}

\vskip1cm\noindent
\textbf{Acknowledgments.}
The author acknowledges support by Deutsche Forschungsgemeinschaft  (DFG) through the Collaborative Research Centre TRR 326 ``Geometry and Arithmetic of Uniformized Structures'', project number 444845124.

\vskip1cm\noindent
\textbf{Declarations.}
The author declares that there is no conflict of interest. Data sharing is not applicable to this article as no datasets were generated
or analysed during the current study.

\bibliographystyle{./meinStil}
\bibliography{./citations}

\small{
{\sc  Universit\"{a}t Heidelberg, Institut f\"{u}r Mathematik, Im Neuenheimer Feld 205, D-69120 Heidelberg, Deutschland}

\textit{E-mail address:} {\tt schmidt@mathi.uni-heidelberg.de}}

\end{document}